\documentclass[leqno,12pt]{article}

\usepackage{epsfig}

\usepackage{amsmath}
\usepackage{amscd}
\usepackage{amsopn}
\usepackage{amsthm}
\usepackage{amsfonts,amssymb}\usepackage{amsfonts,bbm}
\usepackage{latexsym}

\usepackage{color}

\newtheorem{lemma}{LEMMA}
\newtheorem{proposition}[lemma]{PROPOSITION}
\newtheorem{corollary}[lemma]{COROLLARY}
\newtheorem{theorem}[lemma]{THEOREM}
\newtheorem{remark}[lemma]{REMARK}
\newtheorem{remarks}[lemma]{REMARKS}
\newtheorem{example}[lemma]{EXAMPLE}

\newtheorem{definition}[lemma]{DEFINITION}

\newenvironment{knownresult}[1][ASSUMPTION]{\begin{trivlist}
\item[\hskip \labelsep {\bfseries #1.}] \itshape}    {\end{trivlist}}

\newcommand{\real}{\mathbbm{R}}

\newcommand{\nat}{\mathbbm{N}}

\renewcommand{\a}{\alpha}
\renewcommand{\b}{\beta}

\newcommand{\g}{\gamma}

\newcommand{\vp}{\varphi}
\newcommand{\ve}{\varepsilon}

\newcommand{\reald}{{\real^d}}

\newcommand{\on}{\quad\text{ on }}
\newcommand{\und}{\quad\mbox{ and }\quad}
\newcommand{\inv}{^{-1}}
\newcommand{\ov}{\overline}

\newcommand{\V}{\mathcal V}  
\newcommand{\W}{\mathcal W}  
\newcommand{\C}{\mathcal C}

\renewcommand{\H}{{\mathcal H}}
\newcommand{\B}{\mathcal B}

\newcommand{\M}{\mathcal M}

\newcommand{\K}{\mathcal K}
\newcommand{\Q}{\mathcal Q}

\newcommand{\dist}{\mbox{\rm dist}}

\newcommand{\supp}{\operatorname*{supp}}

\newcommand{\convex}{\operatorname*{conv}}

\newcommand{\itemframe}%
{\setlength{\parskip}{10pt}\begin{enumerate} \setlength{\topsep}{10pt}%
\setlength{\itemsep}{15pt}\setlength{\parsep}{5pt}}

\newcommand{\vx}{\ve_x}

\newcommand{\Px}{\mathcal P(X)}

\newcommand{\uc}{{U^c}}

\newcommand{\vc}{{V^c}}
\newcommand{\wc}{{W^c}}

\newcommand{\lam}{\lambda}

\newcommand{\kap}{\operatorname*{cap}}

\newcommand{\ppt}{\mathbbm P=(P_t)_{t>0}}
\newcommand{\vvl}{\mathbbm V=(V_\lambda)_{\lambda>0}}
\newcommand{\liml}{\lim_{\lambda\to \infty}} 
\newcommand{\ev}{E_{\mathbbm V}}  
\newcommand{\es}{E_{\mathbbm P}}  
\newcommand{\lvl}{\lambda V_\lambda}

\newcommand{\du}{\delta_U}
\newcommand{\gr}{G_0 }

\date{}

\title{Hunt's hypothesis (H) and triangle property\\  of  the Green function}
\author{Wolfhard Hansen and Ivan Netuka}

\begin{document}
\maketitle 

\begin{abstract}

Let $X$ be a locally compact abelian group with countable base and let $\mathcal W$ be
a convex cone of positive numerical functions on $X$ which is invariant under the group action
and such that  $(X,\mathcal W)$ is a balayage space or (equivalently, if $1\in \mathcal W$) such that  $\mathcal W$
is the set of excessive functions of a Hunt process on~$X$,  
$\mathcal W$  separates points, every function in $\mathcal W$ is the supremum of its
continuous minorants in $\mathcal W$, and there exist  strictly positive continuous  $u,v\in \mathcal W$ 
such that $u/v\to 0$  at infinity.

 Assuming that there is a~Green function 
$G>0$ for $X$ which locally satisfies the triangle inequality $G(x,z)\wedge G(y,z)\le C G(x,y)$
(true for many L\'evy processes), 
it is shown that Hunt's hypothesis (H) holds, that is, every semipolar set is polar.


 {

  }

\end{abstract}

The purpose of this short paper is to show that in the settings considered 
in \cite{H-fuku,H-wiener,HN-unavoidable,HN-harnack} 
Hunt's hypothesis (H) holds, that is, semipolar sets are polar provided the underlying space $X$ is
an abelian group and the set $\W$ of positive hyperharmonic functions on $X$
(the set of excessive functions of a~corresponding Hunt process) is invariant
under the group action. The essential property we use is a~local triangle
property of a~Green function for $(X,\W)$.
Our results constitute a~contribution
to the long-lasting discussion of Getoor's conjecture, that is, of the 
validity of (H) for all ``reasonable'' L\'evy processes (see \cite{glover-rao, Hu-Sun-Zhang} and Example \ref{levy}).

Let $X$ be a~locally compact space with countable base.
Let $\C(X)$ denote the set of all continuous real functions on $X$
and let $\B(X)$ be the set of all Borel measurable numerical functions on $X$.
The set of all (positive) Radon measures on $X$ will be denoted by $\M(X)$.

Moreover, let $\W$ be a~convex cone of positive lower 
semicontinuous numerical functions on~$X$ such that  $(X,\W)$ is a~balayage space
(see \cite{BH}, \cite{H-course}  or \cite[Appendix]{HN-unavoidable}). In particular, the following holds: 
 
\begin{itemize} 
\item[\rm (C)]
 $\W$ linearly separates the points of $X$, for every $w\in \W$,      
\[
              w=\sup\{v\in\W\cap \C(X)\colon v\le w\}, 
\]
and there are strictly positive $u,v\in\W\cap \C(X)$ such that $u/v\to 0$  at~infinity.   
\end{itemize} 

\begin{remarks}{\rm
1. If $1\in \W$, then there exists a~Hunt process $\mathfrak X$ 
on~$X$ such that $\W$ is the set $\es$ of excessive functions for the transition semigroup 
$\mathbbm P=(P_t)_{t>0}$ of~$\mathfrak X$ (see \cite[Proposition
1.2.1]{H-course} and \cite[IV.8.1]{BH}), that is,
\[
             \W=\{v\in \B^+(X)\colon \sup\nolimits_{t>0} P_tv=v\}.
\]

2. Let us note that the condition $1\in \W$ is not very
restrictive. Indeed, if $(X,\W)$ is a~balayage space, 
$w_0\in \W\cap C(X)$ is strictly positive, and $\widetilde \W:=\{w/w_0\colon w\in \W\}$,   
then $(X,\widetilde \W)$ is a~balayage space such that $1\in \widetilde \W$,
and results for $(X,\widetilde \W)$ yield results for~$(X,\W)$.

3. Moreover, given any sub-Markov right-continuous semigroup $\mathbbm P=(P_t)_{t>0}$  on~$X$
such that (C) is satisfied by its convex cone $\es$ of excessive functions, $(X,\es)$ is a~balayage space, 
and $\mathbbm P$ is the transition semigroup of a~Hunt process (see  \cite[Corollary 2.3.8]{H-course}
or \cite[Corollary A.5]{HN-unavoidable}).   
}              
 \end{remarks}

Let us recall that, for all $A\subset X$ and $u\in \W$, the function $R_u^A$ is the infimum  
of all $v\in \W$ such that $v\ge u$ on $A$, and $\hat R_u^A(x):=\liminf_{y\to x} R_u^A(y)$, $x\in X$.

 A set $P$ in $X$ is \emph{polar}, if $\hat R_v^P=0$ for some (every) function $v>0$ in $\W$. A set $T$ in $X$
  is \emph{totally thin}, if $\hat R_v^T<v$ for some $v\in\W$, and \emph{semipolar},
if it is a~countable union of totally thin sets. For example, the sets $\{\hat R_u^A<R_u^A\}$,
$A\subset X$, $u\in \W$, are semipolar (and subsets of $A\cap \partial A$; see \cite[VI.5.11 and VI.2.3]{BH}). 

A function $h\in \H^+(X)$ is \emph{harmonic} on an open set $U$ in $X$ if $h|_U\in \C(U)$ and
$\int h\,d\vx^\vc=h(x)$, for all $x\in U$ and open   $V$  such that $x\in V$ and
$\ov V$ is compact in $U$ (the measures $\vx^\vc$ are given by $\int u \,d\vx^\vc=R_u^\vc(x)$, $u\in \W$).

\begin{knownresult}[ASSUMPTION A] 
Let us assume that  $G\colon X\times X\to
(0,\infty]$ is  a~Borel measurable function, 
$G=\infty$ on the diagonal, $G<\infty$ off the diagonal, 
such that $G$ is a~Green function for $(X,\W)$, that is, 
the following holds {\rm(}see {\rm \cite{BH,H-course}} for the definition of potentials for $(X,\W)${\rm)}:
\begin{itemize} 
\item [\rm (i)] 
   For every $y\in X$,  $ G(\cdot,y)$ is a~potential which is harmonic on $X\setminus \{y\}$.   
\item [\rm (ii)] 
For every potential $p$ on $X$, there exists a~measure $\mu$ on $X$    such that 
\begin{equation}\label{p-rep}
p=G\mu:=\int G(\cdot,y)\,d\mu(y).
\end{equation} 
\end{itemize} 
\end{knownresult} 

\begin{remarks}{\rm
1. Having (i), each of the following properties implies (ii).
\begin{itemize}
\item
$G$ is locally bounded off the diagonal, each function $G(x,\cdot)$ is lower semicontinuous 
on $X$ and continuous on $X\setminus \{x\}$, and there exists a~measure~$\nu$ on~$X$ such that $G\nu\in \C(X)$ 
and $\nu(U)>0$ for every  finely open  $U\ne \emptyset$ (see \cite[Theorem~4.1]{HN-representation}).
\item
$G$ is lower semicontinuous on $X\times X$, continuous outside the diagonal,
and $\W=\es$ for some sub-Markov semigroup $\mathbbm P=(P_t)_{t>0}$ such that 
the potential kernel $V_0:=\int_0^\infty P_t\,dt$   is 
proper, and there is a~measure~$\mu$ on $X$ such that $V_0f:=\int G(\cdot, y)\,d\mu(y)$
(see \cite[p.~114]{maagli-87}, where (c) follows from \cite[VI.2.6 and III.6.6]{BH}).
\end{itemize} 

2. 
The measure in (\ref{p-rep}) is uniquely determined and, given any measure~$\mu$
on~$X$ such that $p:=G\mu$ is a~potential, the complement of the support of $\mu$ is the largest open
set, where $p$ is harmonic (see, for example, \cite[Proposition 5.2 and Lemma~2.1]{HN-representation}).
%
}
\end{remarks}

\begin{example}\label{levy} {\rm
Let  $\mathbbm P=(P_t)_{t>0}$ be a right continuous sub-Markov semigroup on~$\reald$, $d\ge 1$,
such that the potential kernel $V_0:=\int_0^\infty P_t\,dt$ is given by $V_0f=G_0\ast f$, 
 where $G_0=g(|\cdot|)$  and $g\colon [0,\infty)\to (0,\infty]$ is decreasing with 
 \hbox{$\lim_{r\to 0} g(r)=g(0)=\infty$}, $\lim_{r\to\infty} g(r)=0$,  $\int_0^1 g(r) r^{d-1}\,dr<\infty$, and $g(r)\le C g(2r)$
for small $r>0$.

Then $(\reald, \es)$ is a balayage space such that $G(x,y):=G_0(x-y)$ satisfies Assumption A as well as 
the following Assumption B   (see \cite[Section 6]{HN-unavoidable}
and \cite{H-fuku}; cf.\ \cite{H-wiener} for more general L\'evy processes).
}
\end{example}

\begin{knownresult}[ASSUMPTION B] 
We assume, in addition, that $G$ has the \emph{local triangle
  property}, that is, $X$ is covered by open sets $U$ for which there exists a constant $C>0$
such that 
\begin{equation}\label{tri}
         G(x,z)\wedge G(y,z)\le C G(x,y), \qquad x,y,z\in U.
\end{equation} 
\end{knownresult}

\begin{proposition}[Continuity principle of Evans-Vasilesco]\label{evans-vasilesco}
Let $\mu$ be a~measure on $X$, $A:=\supp (\mu)$ and $x_0\in A$ such that $G\mu$ is a~potential
and $(G\mu)|_A$ is continuous at~$x_0$. Then $G\mu$ is continuous at~$x_0$.
\end{proposition} 

{\sl Proof} (cf.~the proof of \cite[V.4.11]{BH}).  If $G\mu(x_0)=\infty$, then $G\mu$ is continuous at~$x_0$.
So let $G\mu(x_0)<\infty$. Let $U$ be an open neighborhood of $x_0$  
such that (\ref{tri}) holds. If $K$ is a
compact neighborhood of $x_0$ in $U$ and $\mu':=1_K\mu$, $\mu'':=1_{K^c}\mu$,
then $G\mu=G\mu'+G\mu''$, and $G\mu''$ is continuous at~$x_0$,  by
\cite[Lemma~2.1]{HN-representation}. 
Hence we may assume without loss of generality that the support $A$ of
$\mu$ is a~non-empty compact in $U$. 

By (\ref{tri}), $G(y,x)\le C G(x,y) $,  $x,y\in U$  (take $z=x$). Further, for all $x,y,z\in U$,   
\[
G(x,y)\inv \le C (G(x,z)\wedge G(y,z))\inv \le C(G(x,z)\inv + G(y,z)\inv).
\]
Therefore $(x,y)\mapsto G(x,y)\inv + G(y,x)\inv$ is a quasi-metric on $U\times U$ and,
  by \cite[Proposition 14.5]{heinonen}, there exist $c\ge 1$, $\g>0$ and a~metric $\rho$
for  $U$ such that 
\[   
      c\inv \rho^{-\g}\le G\le c \rho^{-\g} \on U\times U.
\]   
For   $x\in U$, let $y_x\in A$ be such that 
\[
 \rho(x,y_x)=\min\{\rho(x,y)\colon y\in A\}.
\]
Then, for every $y\in A$,
\[
  \rho(y_x,y)\le \rho(y_x,x)+\rho(x,y)\le 2\rho(x,y),
\]
and hence, for all measures $\nu$ on  $A$ and $x\in U$,
\begin{equation}\label{essential}
    G\nu(x)\le c \int \rho(x,y)^{-\g}\,d\nu(y)\le 2^\g c \int
    \rho(y_x,y)^{-\g}\,d\nu(y) \le 2^\g c^2G\nu(y_x).
\end{equation} 
Let us now fix $\ve>0$. Since $G\mu(x_0)<\infty$, there exists $r>0$ 
such that $\ov B(x_0,r)\subset U$ (where, of course, $B(x_0,r):=\{y\in U\colon \rho(y,x_0)<r\}$) and 
$\nu:=1_{B(x_0,r)}\mu$ satisfies $ G\nu(x_0)<\ve$. Then, again by
\cite[Lemma 2.1]{HN-representation}, 
$G(\mu-\nu)$ is continuous at~$x_0$, and hence $(G\nu)|_A$ is continuous at~$x_0$ as well.
So there exists $\delta>0$ such that 
\begin{eqnarray*} 
   |G(\mu-\nu)-G(\mu-\nu)(x_0)|&<&\ve \on  B(x_0,\delta),\\
   G\nu&<&\ve \on B(x_0,2\delta)\cap A.
\end{eqnarray*} 
If $x\in B(x_0,\delta)$, then $\rho(x,y_x)\le \rho(x,x_0)<\delta$,
hence 
$\rho(y_x,x_0)\le \rho(y_x,x)+\rho(x,x_0)<2\delta$,
and therefore, by (\ref{essential}), 
\begin{equation*} 
   |G\nu(x)-G\nu(x_0)|\le G\nu(x)+G\nu(x_0)\le 2^\g c^2G\nu(y_x)+G\nu(x_0)<(2^\g c^2+1)\ve.
\end{equation*} 
Thus $|G\mu-G\mu(x_0)|<(2^\g c^2 +2)\ve$ on $B(x_0,\delta)$.
\hfill $\square$

\begin{remark}{\rm
In the proof of Proposition \ref{evans-vasilesco} it would, of course, be sufficient to know
that $G\approx g\circ \rho$  for some metric $\rho$ on $U$ and some  decreasing 
function~$g$ having the doubling property near $0$ (cf.\ also  \cite[Proposition 1.7]{H-wiener}   
for equivalences).
}
\end{remark}

\begin{proposition}\label{axiom-D}
For every real potential $p$ on $X$, there exists a~sequence $(p_n)$
of continuous real potentials on $X$ such that $\sum_{n\in\nat} p_n=p$
and the superharmonic supports of $p_n$, $n\in\nat$, are pairwise disjoint.
\end{proposition} 

{\sl Proof} (cf.  the proof of \cite[V.4.12]{BH}). 
There exists a~measure $\mu$ on $X$ such that $G\mu=p$.
By Lusin's theorem, there exists a~sequence $(K_n)$ of pairwise disjoint compacts in $X$
such that $\mu(X\setminus \bigcup_{n\in\nat} K_n)=0$ and $p|_{K_n}$ is continuous for every $n\in\nat$.
Let 
\[ 
\mu_n:=1_{K_n} \mu \und p_n:=G\mu_n.
\]
Then, of course, $\sum_{n\in\nat} p_n=p$. For every $n\in\nat$, $p_n|_{K_n}$ is continuous, 
since $p|_{K_n}$ is continuous and both $p_n$, $\sum_{m\in \nat, m\ne n} p_m$ are lower semicontinuous.
Thus $p_n\in \C(X)$, by Proposition \ref{evans-vasilesco}. \hfill $\square$

\begin{corollary}[Domination principle]\label{domination}
Let $\mu$ be a~measure on $X$ and let $A$ be a~Borel measurable subset of $X$
such that $G\mu$ is a~real potential on~$X$ and $\mu(A^c)=0$. Then
\[ 
R_{G\mu}^A=G\mu.
\]
\end{corollary} 

\begin{proof} See the proofs of \cite[V.4.13 and V.4.14]{BH}.
\end{proof}

\begin{knownresult}[ASSUMPTION C]
 From now on let us assume, in addition, that $X$ is an abelian topological
group and that $\W$ is invariant under the translations $x\mapsto
x+y$, $y\in X$. 
\end{knownresult}

\begin{proposition}\label{har-com}
Suppose that $v\in \W$, $w\in \W\cap \C(X)$, and $v\le w$. 
Then, for every $A\subset X$, the function $\hat R_v^A$ is harmonic in $X\setminus \ov A$.
\end{proposition} 

\begin{proof} By \cite[VI.2.2]{BH}, we may assume that $A$ is a~Borel set 
(there exists a~$G_\delta$-set~$A'$ such that $A\subset A'\subset \ov
  A$ and $\hat R_v^{A'}=\hat  R_v^A$). Let  $u:=R_v^A$. By
  \cite[VI.2.4]{BH}, 
$  u=\hat u$ on~$X\setminus A$. In particular, $u$ is Borel measurable
(note  that $u=v$ on $A$).
By \cite[VI.2.6]{BH}, $u$ is harmonic on $X\setminus \ov A$.
We recall that, for general balayage spaces, this does not imply
that $\hat u$ is harmonic on~$X\setminus \ov A$, not even if $A$ is compact  (see \cite[V.9.1]{BH}). 

Let $x\in X\setminus \ov A$ and let $U$ be a
  relatively compact open neighborhood of $x$ such that $\ov U\cap \ov
  A=\emptyset$. We intend to show that
\begin{equation}\label{hau}
                 \int \hat u\,d\vx^\uc=\hat u(x). 
\end{equation} 
To that end we fix  a~relatively compact  open neighborhood $V$ of $0$ such that 
$\ov U+V$ does not intersect $\ov A$.
Moreover, let $\lambda$ be a~measure on $X$ such that $G\lambda$ is a~continuous real potential
which is strict. By  \cite[VI.8.2 and VI.5.15]{BH}, $\lambda$
charges every Borel measurable non-empty finely open set, but does not
charge semipolar sets. We define $\nu:=1_V\lambda$ and
\[
    \mu:=\vx^\uc\ast \nu,
\]
that is, for every  Borel measurable set $B$ in $X$,
\[
 \mu(B)=\int \vx^\uc(B-y)\,d\nu(y)=\int \nu(B-y)\,d\vx^\uc(y).
\]
In particular, $\mu$ does not charge semipolar sets.

By translation invariance, the functions $y\mapsto u(y+z)$, $z\in
X$, are harmonic on~$X\setminus \ov{(A+z)}$, and hence 
\[ 
\int u\,d\mu=\int \int   u(y+z) \, d\vx^\uc(y) \,d\nu(z) 
                    =\int u(x+z)\,d\nu(z).
\] 
The  set $\{\hat u<u\}$ is semipolar, by \cite[VI.5.11]{BH},
and  hence the integrals do not change, if $u$ is replaced by $\hat u$,
that is,
\[
    \int \int   \hat u(y+z) \, d\vx^\uc(y) \,d\nu(z) =\int \hat
    u\,d\mu=\int \hat u(x+z)\,d\nu(z).
\]
Since $\hat u\in\W$, we know that $\int \hat u(y+z)\,d\vx^\uc(y)\le \hat u(x+z)$,
for every $z\in X$. Therefore
\begin{equation}\label{ae}
u_x(z):=\int \hat u(y+z)\,d\vx^\uc(y)= \hat u(x+z)
\end{equation}
for $\nu$-almost every $z\in X$, where also $u_x\in \W$.
Since $\nu$ charges every non-empty finely open set in $V$,
we conclude that (\ref{ae}) holds for every $z\in V$.
In particular, (\ref{hau})~holds.    
\end{proof} 

\begin{corollary}\label{unique}
For every relatively compact set $A$ in $X$, there exists a~unique measure $\mu$
on $\ov A$ such that  $\hat R_1^A=G\mu$.
\end{corollary}

\begin{corollary}\label{H} 
Every semipolar set in $X$ is polar.
\end{corollary}

\begin{proof} 
Let $K$ be a~compact, nonpolar set in $X$. Then $\hat R_1^K\ne 0$ and, by Proposition~\ref{har-com},  
$\hat R_1^K$ is harmonic on $K^c$. Hence, by Proposition ~\ref{axiom-D}, there exists a~continuous potential 
$p\ne 0$  on $X$ such that $\hat R_1^K-p\in \W$. 
Then both $p$ and $\hat R_1^K-p$ are  harmonic on $K^c$.  By \cite[VI.5.15]{BH}, this implies that $K$ is not semipolar.
\end{proof} 




\bibliographystyle{plain} 
\def\cprime{$'$} \def\cprime{$'$}

{\small \noindent 
Wolfhard Hansen,
Fakult\"at f\"ur Mathematik,
Universit\"at Bielefeld,
33501 Bielefeld, Germany, e-mail:
 hansen$@$math.uni-bielefeld.de}\\
{\small \noindent Ivan Netuka,
Charles University,
Faculty of Mathematics and Physics,
Mathematical Institute,
 Sokolovsk\'a 83,
 186 75 Praha 8, Czech Republic, email:
netuka@karlin.mff.cuni.cz}

\end{document}